\documentclass{article}

\usepackage{amsmath,amssymb,amsthm,dsfont,mathrsfs,stmaryrd}
\usepackage{geometry}
\usepackage{hyperref}

\newcommand\1{{\mathds{1}}}
\renewcommand\ae{{a.\@e.\@}}
\newcommand\BV{{\rm BV}}
\newcommand*{\coleq}{\mathrel{\vcenter{\baselineskip0.5ex\lineskiplimit0pt\hbox{\normalsize.}\hbox{\normalsize.}}}=}
\newcommand\D{{\rm D}}
\newcommand\DM{{\mathcal{DM}^\infty(\Omega,\R^n)}}
\newcommand\DMloc{{\mathcal{DM}^\infty_\loc(\Omega,\R^n)}}
\renewcommand\d{{\rm d}}
\newcommand\Div{{\rm div}}
\newcommand\dx{{\,\d x}}
\newcommand\ecke{\mathop{\hbox{\vrule height 7pt width .3pt depth 0pt \vrule height .3pt width 5pt depth 0pt}}\nolimits}
\renewcommand\L{{\rm L}}
\newcommand\loc{{\rm loc}}
\newcommand\lrb[1]{{\llbracket#1\rrbracket}}
\renewcommand\H{{\mathcal{H}}}
\newcommand\Ln{{\mathcal{L}^n}}
\newcommand\N{{\mathds{N}}}
\newcommand\p{\varphi}
\newcommand\qq{\qquad}
\newcommand\R{{\mathds{R}}}
\newcommand\Sinfty{{S^\infty(\Omega,\R^n)}}
\newcommand\W{{\rm W}}

\newtheorem{theorem}{Theorem}[section]
\newtheorem{lemma}[theorem]{Lemma}
\newtheorem{proposition}[theorem]{Proposition}
\newtheorem{definition}[theorem]{Definition}
\newtheorem*{remark}{Remark}

\begin{document}

\title{An Anzellotti type pairing for divergence-measure fields\\
  and a notion of weakly super-$1$-harmonic functions}

\author{
  Christoph Scheven\footnote{
    Fakult\"at f\"ur Mathematik, Universit\"at Duisburg-Essen,
    Thea-Leymann-Str. 9, 45127 Essen, Germany.\newline
    Email address: \href{mailto:christoph.scheven@uni-due.de}
      {\tt christoph.scheven@uni-due.de}.
    URL: \href{http://www.uni-due.de/mathematik/ag_scheven}
      {\tt http:/\!/www.uni-due.de/mathematik/ag\_scheven}.
    }
  \and
  Thomas Schmidt\footnote{
    Fachbereich Mathematik, Universit\"at Hamburg,
    Bundesstr. 55, 20146 Hamburg, Germany.\newline
    Email address: \href{mailto:thomas.schmidt@math.uni-hamburg.de}
      {\tt thomas.schmidt@math.uni-hamburg.de}.
    URL: \href{http://www.math.uni-hamburg.de/home/schmidt}
      {\tt http:/\!/www.math.uni-hamburg.de/home/schmidt}.
    }
    \footnote{Former institution: Department Mathematik,
      Friedrich-Alexander-Universit\"at Erlangen-N\"urnberg.
    }
  }

\date{}

\maketitle

\begin{abstract}
  We study generalized products of divergence-measure fields and
  gradient measures of $\BV$ functions. These products depend on the
  choice of a representative of the $\BV$ function, and here we single
  out a specific choice which is suitable in order to define and
  investigate a notion of weak supersolutions for the $1$-Laplace
  operator.
\end{abstract}

\section{Introduction}

For a positive integer $n$ and an open set $\Omega$ in $\R^n$, we consider the
$1$-Laplace equation
\begin{equation}\label{eq:Delta1}
  \Div\,\frac{\D u}{|\D u|}\equiv0\qq\text{on }\Omega\,.
\end{equation}
This equation formally arises as the Euler equation of the total variation and
is naturally posed for functions $u\colon\Omega\to\R$ of locally bounded
variation whose gradient $\D u$ is merely a measure. In order to make sense of
\eqref{eq:Delta1} in this setting it has become standard
\cite{KOHTEM83,DEMENGEL99,DEMENGEL04,BELCASNOV05,MERSEGTRO08,MERSEGTRO09}
to work with a generalized product, which has been studied systematically by
Anzellotti \cite{ANZELLOTTI83a,ANZELLOTTI83b}. The product is defined for
$u\in\BV_\loc(\Omega)$ and $\sigma\in\L^\infty_\loc(\Omega,\R^n)$ with vanishing
distributional divergence $\Div\,\sigma$ as the distribution
\[
  \lrb{\sigma,\D u}
  \coleq\Div(u\sigma)\in{\mathscr D}'(\Omega)\,,
\]
and in fact the pairing $\lrb{\sigma,\D u}$ turns out to be a signed Radon
measure on $\Omega$. By requiring $\|\sigma\|_{\L^\infty(\Omega,\R^n)}\le1$ and
$\lrb{\sigma,\D u}=|\D u|$ one can now phrase precisely what it means that
$\sigma$ takes over the role of $\frac{\D u}{|\D u|}$, and whenever there exists
some $\sigma$ with all these properties (including $\Div\,\sigma\equiv0$), one
calls $u\in\BV_\loc(\Omega)$ a $\BV$ solution of \eqref{eq:Delta1} or a weakly
$1$-harmonic function on $\Omega$. In a similar vein, variants of the pairing
$\lrb{\sigma,\D u}$ can be used to define $\BV$ solutions $u$ of
$\Div\frac{\D u}{|\D u|}=f$ for right-hand sides $f\in\L^n_\loc(\Omega)$ and to
explain $\BV\cap\L^\infty$ solutions $u$ of $\Div\frac{\D u}{|\D u|}=f$ even for
arbitrary $f\in\L^1_\loc(\Omega)$.

In this note we deal with a notion of \emph{super}solutions of \eqref{eq:Delta1}
or --- this is essentially equivalent --- of solutions of
$\Div\frac{\D u}{|\D u|}={-}\mu$ with a Radon measure $\mu$ on $\Omega$. To this
end, we first collect some preliminaries in Section \ref{sec:pre}. Then, in
Section \ref{sec:Anzellotti}, we consider generalized pairings, which make sense
even for $\L^\infty$ divergence-measure fields $\sigma$, but require precise
evaluations of $u\in\BV_\loc(\Omega)\cap\L^\infty_\loc(\Omega)$ up to sets of
zero $(n{-}1)$-dimensional Hausdorff measure $\H^{n-1}$. We mainly investigate a
pairing $\lrb{\sigma,\D u^+}$, which is built with a specific $\H^{n-1}$-{\ae}
defined representative $u^+$ of $u$ and which does not seem to have been
considered before (while a similar pairing with the mean-value representative
$u^\ast$ already occurred in \cite[Theorem 3.2]{CHEFRI99} and
\cite[Appendix A]{MERSEGTRO09}). Adapting the approach in
\cite[Section 5]{BECSCH15}, we moreover deal with an up-to-the-boundary pairing
$\lrb{\sigma,\D u^+}_{u_0}$, which accounts for a boundary datum $u_0$. In
Section \ref{sec:super-1} we employ the local pairing $\lrb{\sigma,\D u^+}$ in
order to introduce a notion of weakly super-$1$-harmonic functions, and we prove
a compactness statement which crucially depends on the choice of the
representative $u^+$. Finally, Section \ref{sec:super-1-Dir} is concerned with a
refined notion of super-$1$-harmonicity which incorporates Dirichlet boundary
values. This last notion is based on (a modification of) the pairing
$\lrb{\sigma,\D u^+}_{u_0}$.

We emphasize that the proofs, which are omitted in this announcement, can be
found in the companion paper \cite{SCHSCH16super1}, where we also provide a more
detailed study of pairings and supersolutions together with adaptions to the
case of the minimal surface equation. Furthermore, in our forthcoming work
\cite{SCHSCH16obst}, we will discuss connections with obstacle problems and
convex duality.

\section{Preliminaries}\label{sec:pre}

\noindent\textbf{\boldmath$\L^\infty$ divergence-measure fields.} We
record two results related to the classes
\begin{gather*}
  \DMloc
  \coleq\{\sigma\in\L^\infty_\loc(\Omega,\R^n)\,:\,
  \Div\,\sigma\text{ exists as a signed Radon measure on }\Omega\}\,,\\
  \DM
  \coleq\{\sigma\in\L^\infty(\Omega,\R^n)\,:\,
  \Div\,\sigma\text{ exists as a finite signed Borel measure on }\Omega\}\,.
\end{gather*}

\begin{lemma}[absolute-continuity property for divergences of
  $\L^\infty$ vector fields]\label{lem:Chen-Frid}
  Consider $\sigma\in\DMloc$. Then, for every Borel set $A\subset\Omega$ with
  $\mathcal{H}^{n-1}(A)=0$, we have $|\Div\,\sigma|(A)=0$. 
\end{lemma}

Lemma \ref{lem:Chen-Frid} has been proved by Chen \& Frid
\cite[Proposition 3.1]{CHEFRI99}.

\begin{lemma}[finiteness of divergences with a sign]
    \label{lem:signed-divs-finite}
  If $\Omega$ is bounded with $\H^{n-1}(\partial\Omega)<\infty$ and
  $\sigma\in\L^\infty(\Omega,\R^n)$ satisfies $\Div\,\sigma\le0$ in
  ${\mathscr D}'(\Omega)$, then we necessarily have $\sigma\in\DM$. Moreover,
  there holds $({-}\Div\,\sigma)(\Omega)\le\frac{n\omega_n}{\omega_{n-1}}
  \|\sigma\|_{\L^\infty(\Omega,\R^n)}\H^{n-1}(\partial\Omega)$ with
  the volume $\omega_n$ of the unit ball in $\R^n$.
\end{lemma}

Indeed, it follows from the Riesz representation theorem that $\Div\,\sigma$ in
Lemma \ref{lem:signed-divs-finite} is a Radon measure. The finiteness of this
measure will be established in \cite{SCHSCH16super1} by a reasoning based on the
divergence theorem.

\vspace{3ex}

\noindent\textbf{\boldmath$\BV$-functions.} We mostly follow the terminology of
\cite{AMBFUSPAL00}, but briefly comment on additional conventions and results.
For $u\in\BV(\Omega)$, we recall that $\H^{n-1}$-{\ae} point in $\Omega$ is
either a Lebesgue point (also called an approximate continuity point) or an
approximate jump point of $u$; compare \cite[Sections 3.6,\,3.7]{AMBFUSPAL00}.
We write $u^+$ for the $\H^{n-1}$-{\ae} defined representative of $u$ which takes
the Lebesgue values in the Lebesgue points and the larger of the two jump values
in the approximate jump points. Correspondingly, $u^-$ takes on the lesser jump
values, and we set $u^\ast\coleq\frac12(u^+{+}u^-)$. Finally, if $\Omega$ has
finite perimeter in $\R^n$, we write $u_{\partial^\ast\Omega}^{\rm int}$ for
the $\H^{n-1}$-{\ae} defined interior trace of $u$ on the reduced boundary
$\partial^\ast\Omega$ of $\Omega$ (compare
\cite[Sections 3.3,\,3.5,\,3.7]{AMBFUSPAL00}), and in the case that
$\H^{n-1}(\partial\Omega\setminus\partial^\ast\Omega)=0$ we also denote the
same trace by $u_{\partial\Omega}^{\rm int}$.

The following two lemmas are crucial for our purposes. The first one extends
\cite[Proposition 3.62]{AMBFUSPAL00} and is obtained by essentially the same
reasoning. The second one follows by combining
\cite[Theorem 2.5]{CARDALLEAPAS88} and \cite[Lemma 1.5, Section 6]{DALMASO83};
compare also \cite[Sections 4,\,10]{FEDZIE72}.

\begin{lemma}[$\BV$ extension by zero]
  If\label{lem:BV-ext} we have $\H^{n-1}(\partial\Omega)<\infty$, then for every
  $u\in\BV(\Omega)\cap\L^\infty(\Omega)$ we have $\1_\Omega u\in\BV(\R^n)$ and
  $|\D(\1_\Omega u)|(\partial\Omega)
  \le\frac{n\omega_n}{\omega_{n-1}}\|u\|_{\L^\infty(\Omega)}\H^{n-1}(\partial\Omega)$.
  In particular, $\Omega$ is a set of finite perimeter in $\R^n$, and
  $u_{\partial^\ast\Omega}^{\rm int}$ is well-defined.
\end{lemma}

\begin{lemma}[$\H^{n-1}$-{\ae} approximation of a $\BV$ function from
    above]
  For\label{lem:approx-BV-above} every
  $u\in\BV_\loc(\Omega)\cap\L^\infty_\loc(\Omega)$ there exist
  $v_\ell\in\W^{1,1}_\loc(\Omega)\cap\L^\infty_\loc(\Omega)$ such that $v_1\ge
  v_\ell\ge u$ holds $\Ln$-{\ae} on $\Omega$ for every $\ell\in\N$ and such that
  $v_\ell^\ast$ converges $\H^{n-1}$-{\ae} on $\Omega$ to $u^+$.
\end{lemma}

\section{\boldmath Anzellotti type pairings for $\L^\infty$
  divergence-measure fields}\label{sec:Anzellotti}

We first introduce a local pairing of divergence-measure fields and gradient
measures.

\begin{definition}[local pairing]\label{defi:local-pairing}
  Consider $u\in\BV_\loc(\Omega)\cap\L^\infty_\loc(\Omega)$ and
  $\sigma\in\DMloc$. Then --- since Lemma \ref{lem:Chen-Frid} guarantees that
  $u^+$ is $|\Div\,\sigma|$-{\ae} defined --- we can define the distribution
  \[
    \lrb{\sigma,\D u^+}
    \coleq\Div(u\sigma)-u^+\Div\,\sigma
    \in{\mathscr D}'(\Omega)\,.
  \]
\end{definition}

Written out this definition means
\begin{equation}\label{eq:defi-pairing}
  \lrb{\sigma,\D u^+}(\p)
  ={-}\int_\Omega u\sigma\cdot\D\p\dx
  -\int_\Omega\p u^+\,\d(\Div\,\sigma)
  \qq\text{for }\p\in{\mathscr D}(\Omega)\,.
\end{equation}

Next we define a global pairing, which incorporates Dirichlet boundary values
given by a function $u_0$.

\begin{definition}[up-to-the-boundary pairing]\label{defi:global-pairing}
  Consider $u_0\in\W^{1,1}(\Omega)\cap\L^\infty(\Omega)$,
  $u\in\BV(\Omega)\cap\L^\infty(\Omega)$, and $\sigma\in\DM$. Then we define the
  distribution $\lrb{\sigma,\D u^+}_{u_0}\in{\mathscr D}'(\R^n)$ by setting
  \begin{equation}\label{eq:sigmaDu-Dir}
    \lrb{\sigma,\D u^+}_{u_0}(\p)
    \coleq{-}\int_\Omega(u{-}u_0)\sigma\cdot\D\p\dx
    -\int_\Omega\p(u^+{-}u_0^\ast)\,\d(\Div\,\sigma)
    +\int_\Omega\p\sigma\cdot\D u_0\dx
  \end{equation}
  for $\p\in{\mathscr D}(\R^n)$.
\end{definition}

We emphasize that the pairings in Definitions \ref{defi:local-pairing} and
\ref{defi:global-pairing} coincide on $\p$ with compact support in $\Omega$
(since an integration-by-parts then eliminates $u_0$ in \eqref{eq:sigmaDu-Dir}).
However, the up-to-the-boundary pairing stays well-defined even if $\p$ does
\emph{not} vanish on $\partial\Omega$. In addition, we remark that both pairings
can be explained analogously with other representatives of $u$.

In some of the following statements we impose a mild regularity assumption
on $\partial\Omega$, namely we require
\begin{equation}\label{eq:domain-assum}
  \mathcal{H}^{n-1}(\partial\Omega)=\mathbf P(\Omega)<\infty\,,
\end{equation}
where $\mathbf P$ stands for the perimeter. We remark that the condition
\eqref{eq:domain-assum} is equivalent to having $\mathbf P(\Omega)<\infty$ and
$\mathcal{H}^{n-1}(\partial\Omega\setminus\partial^*\Omega)=0$ and also to
having $\1_\Omega\in\BV(\R^n)$ and $|\D\1_\Omega|=\H^{n-1}\ecke\partial\Omega$.
For a more refined discussion we refer to \cite{SCHMIDT15}, where the relevance
of \eqref{eq:domain-assum} for certain approximation results is pointed out.

Two vital properties of the pairing are recorded in the next statements. The
proofs will appear in \cite{SCHSCH16super1}.

\pagebreak[2]

\begin{lemma}[the pairing trivializes on $\W^{1,1}$-functions]\text{}
    \label{lem:W11pairing}
\begin{itemize}
  \item{\rm(local statement)}
  For $u\in\W^{1,1}_\loc(\Omega)\cap\L^\infty_\loc(\Omega)$, $\sigma\in\DMloc$, and
  $\p\in{\mathscr D}(\Omega)$, we have
  \[
    \lrb{\sigma,\D u^+}(\p)=\int_\Omega\p\sigma\cdot\D u\dx\,.
  \]
  \item{\rm(global statement with traces)}
  If $\Omega$ is bounded with \eqref{eq:domain-assum}, then for every
  $\sigma\in\DM$ there exists a uniquely determined normal trace
  $\sigma_{\rm n}^\ast\in\L^\infty(\partial\Omega;\H^{n-1})$ with
  \[
    \|\sigma_{\rm n}^\ast\|_{\L^\infty(\partial\Omega;\H^{n-1})}
    \le\|\sigma\|_{\L^\infty(\Omega,\R^n)}
  \]
  such that for all $u,u_0\in\W^{1,1}(\Omega)\cap\L^\infty(\Omega)$ and
  $\p\in{\mathscr D}(\R^n)$ there holds
  \[
    \lrb{\sigma,\D u^+}_{u_0}(\p)
    =\int_\Omega\p(\sigma\cdot\D u)\dx
    +\int_{\partial\Omega}\p(u{-}u_0)_{\partial\Omega}^{\rm int}\sigma_{\rm n}^\ast\,\d\H^{n-1}\,.
  \]
  \end{itemize}
\end{lemma}

Next we focus on bounded $\sigma$ with $\Div\,\sigma\le0$. By
Lemma \ref{lem:signed-divs-finite} the pairings stay well-defined in this case.

\begin{proposition}[the pairing is a bounded measure]\text{}
    \label{prop:pairing-estimate}
  Fix $\sigma\in\L^\infty(\Omega,\R^n)$ with $\Div\,\sigma\le0$ in
  ${\mathscr D}'(\Omega)$.
  \begin{itemize}
  \item{\rm(local estimate)}
  For $u\in\BV_\loc(\Omega)\cap\L^\infty_\loc(\Omega)$, the distribution
  $\lrb{\sigma,\D u^+}$ is a signed Radon measure on $\Omega$ with
  \[
    |\lrb{\sigma,\D u^+}|
    \le\|\sigma\|_{\L^\infty(\Omega,\R^n)}|\D u|
    \qq\text{on }\Omega\,.
  \]
  \item{\rm(global estimate with equality at the boundary)}
  If\/ $\Omega$ is bounded with \eqref{eq:domain-assum}, for
  $u_0\in\W^{1,1}(\Omega)\cap\L^\infty(\Omega)$ and
  $u\in\BV(\Omega)\cap\L^\infty(\Omega)$ the pairing $\lrb{\sigma,\D u^+}_{u_0}$
  is a finite signed Borel measure on $\R^n$ with
  \begin{equation}\label{eq:boundary-equality}
    \big|\lrb{\sigma,\D u^+}_{u_0}
    -(u{-}u_0)_{\partial\Omega}^{\rm int}\sigma_{\rm n}^\ast\H^{n-1}\ecke\partial\Omega\big|
    \le\|\sigma\|_{\L^\infty(\Omega,\R^n)}|\D u|\ecke\Omega\,.
  \end{equation}
  \end{itemize}
\end{proposition}

\section{\boldmath Weakly super-$1$-harmonic functions}\label{sec:super-1}

We now give a definition of super-$1$-harmonic functions, which employs the
convenient notation
\[
  \Sinfty
  \coleq\{\sigma\in\L^\infty(\Omega,\R^n)\,:\,
  |\sigma|\le1\text{ holds }\Ln\text{-{\ae} on }\Omega\}\,.
\]

\begin{definition}[weakly super-$1$-harmonic]
    \label{defi:super-1}
  We call $u\in\BV_\loc(\Omega)\cap\L^\infty_\loc(\Omega)$ weakly
  super-$1$-harmonic on $\Omega$ if there exists some $\sigma\in\Sinfty$ with
  $\Div\,\sigma\le0$ in ${\mathscr D}'(\Omega)$ and $\lrb{\sigma,\D u^+}=|\D u|$
  on $\Omega$.
\end{definition}

We next provide a compactness result for super-$1$-harmonic functions. We
emphasize that this result does not hold anymore if one replaces $u^+$ by any
other representative of $u$ in the definition. We also remark that the
assumed type of convergence is very natural, and indeed the statement applies to
every \emph{increasing} sequence of super-$1$-harmonic functions which is
bounded in $\BV_\loc(\Omega)$ and $\L^\infty_\loc(\Omega)$.

\begin{theorem}[convergence from below preserves super-$1$-harmonicity]
  Consider a sequence of weakly super-$1$-harmonic functions $u_k$ on $\Omega$.
  If $u_k$ locally weak$\ast$ converges to a limit $u$ both in $\BV_\loc(\Omega)$
  and $\L^\infty_\loc(\Omega)$ and if $u_k\le u$ holds on $\Omega$ for all
  $k\in\N$, then $u$ is weakly super-$1$-harmonic on $\Omega$.
\end{theorem}

\begin{proof}{}
  In view of Definition \ref{defi:super-1} there exist $\sigma_k\in\Sinfty$ with
  $\Div\,\sigma_k\le0$ in ${\mathscr D}'(\Omega)$ and
  $\lrb{\sigma_k,\D u_k^+}=|\D u_k|$ on $\Omega$. Possibly passing to a
  subsequence, we assume that $\sigma_k$ weak$\ast$ converges in
  $\L^\infty(\Omega,\R^n)$ to $\sigma\in\Sinfty$ with $\Div\,\sigma\le0$ in
  ${\mathscr D}'(\Omega)$, and as before we regard $\Div\,\sigma_k$ and
  $\Div\,\sigma$ as non-positive measures on $\Omega$. We fix a non-negative
  $\p\in{\mathscr D}(\Omega)$ and approximations $v_\ell$ of $u$
  with the properties of Lemma \ref{lem:approx-BV-above}. Relying on
  \eqref{eq:defi-pairing}, Lemma \ref{lem:Chen-Frid}, and the dominated
  convergence theorem, we then infer
  \[\begin{aligned}
    \lrb{\sigma,\D u^+}(\p)
    &={-}\int_\Omega u\sigma\cdot\D\p\dx
    +\int_\Omega\p u^+\,\d({-}\Div\,\sigma)\\
    &=\lim_{\ell\to\infty}\bigg[{-}\int_\Omega v_\ell\sigma\cdot\D\p\dx
    +\int_\Omega\p v_\ell^\ast\,\d({-}\Div\,\sigma)\bigg]
    =\lim_{\ell\to\infty}\lrb{\sigma,\D v_\ell^+}(\p)\,.
  \end{aligned}\]
  Since the pairing trivializes on $v_\ell\in\W^{1,1}_\loc(\Omega)$, we can
  exploit the local weak$\ast$ convergence of $\sigma_k$ in
  $\L^\infty_\loc(\Omega,\R^n)$ and the inequalities $v_\ell^\ast\ge u^+\ge u_k^+$
  to arrive at
  \[\begin{aligned}
    \lrb{\sigma,\D v_\ell^+}(\p)
    =\lim_{k\to\infty}\lrb{\sigma_k,\D v_\ell^+}(\p)
    &={-}\lim_{k\to\infty}\int_\Omega v_\ell\sigma_k\cdot\D\p\dx
    +\lim_{k\to\infty}\int_\Omega\p v_\ell^\ast\,\d({-}\Div\,\sigma_k)\\
    &\ge{-}\int_\Omega v_\ell\sigma\cdot\D\p\dx
    +\liminf_{k\to\infty}\int_\Omega\p u_k^+\,\d({-}\Div\,\sigma_k)\,.
  \end{aligned}\]
  Next we rely in turn on the dominated convergence theorem, on the observation
  that $u_k\sigma_k$ locally weakly converges to $u\sigma$ in
  $\L^1_\loc(\Omega,\R^n)$, on the definition in \eqref{eq:defi-pairing}, on the
  coupling $\lrb{\sigma_k,\D u_k^+}=|\D u_k|$, and finally on a lower
  semicontinuity property of the total variation. In this way, we deduce
  \begin{multline*}
    \lim_{\ell\to\infty}\bigg[{-}\int_\Omega v_\ell\sigma\cdot\D\p\dx
    +\liminf_{k\to\infty}\int_\Omega\p u_k^+\,\d({-}\Div\,\sigma_k)\bigg]\\
    \begin{aligned}
      &={-}\int_\Omega u\sigma\cdot\D\p\dx
      +\liminf_{k\to\infty}\int_\Omega\p u_k^+\,\d({-}\Div\,\sigma_k)\\
      &=\liminf_{k\to\infty}\bigg[{-}\int_\Omega u_k\sigma_k\cdot\D\p\dx
      +\int_\Omega\p u_k^+\,\d({-}\Div\,\sigma_k)\bigg]\\
      &=\liminf_{k\to\infty}\lrb{\sigma_k,\D u_k^+}(\p)
      =\liminf_{k\to\infty}\int_\Omega\p\,\d|\D u_k|\\
      &\ge\int_\Omega\p\,\d|\D u|\,.
  \end{aligned}\hspace{2.5cm}\end{multline*}
  Collecting the estimates, we arrive at the inequality
  $\lrb{\sigma,\D u^+}\ge|\D u|$ of measures on $\Omega$. Since the opposite
  inequality is generally valid by Proposition \ref{prop:pairing-estimate},
  we infer that $u$ is weakly super-$1$-harmonic on $\Omega$.
\end{proof}

\section{\boldmath Super-$1$-harmonic functions with respect to Dirichlet data}\label{sec:super-1-Dir}

Finally, we introduce a concept of super-$1$-harmonic functions with respect to
a generalized Dirichlet boundary datum. In \cite{SCHSCH16obst} we will show that
this notion is useful in connection with obstacle problems.

Concretely, consider a bounded $\Omega$ with \eqref{eq:domain-assum},
$u_0\in\W^{1,1}(\Omega)\cap\L^\infty(\Omega)$, and
$u\in\BV(\Omega)\cap\L^\infty(\Omega)$. We then extend the measure $\D u$ on
$\Omega$ to a measure $\D_{u_0}u$ on $\overline\Omega$ which takes into account
the possible deviation of $u_{\partial\Omega}^{\rm int}$ from the boundary datum
$(u_0)_{\partial\Omega}^{\rm int}$. To this end, writing $\nu_\Omega$ for the inward
unit normal of $\Omega$, we set
\[
  \D_{u_0}u
   \coleq\D u\ecke\Omega
   +(u{-}u_0)_{\partial\Omega}^{\rm int}\nu_\Omega\H^{n-1}\ecke\partial\Omega\,.
\]

Now, for $\sigma\in\Sinfty$ with $\Div\,\sigma\le0$ in ${\mathscr D}'(\Omega)$
--- which is meant to potentially satisfy a coupling like
$\lrb{\sigma,\D u^+}_{u_0}=|\D_{u_0}u|$ on $\overline\Omega$ --- we adopt the
viewpoint that $\sigma_{\rm n}^\ast$ should typically equal the constant $1$.
If this is not the case, we compensate for this defect by extending
$({-}\Div\,\sigma)$ to a measure on $\overline\Omega$ with
\begin{equation}\label{eq:div-sigma-boundary}
  (-\Div\,\sigma)\ecke\partial\Omega
  \coleq(1{-}\sigma_{\rm n}^\ast)\H^{n-1}\ecke\partial\Omega\,.
\end{equation}
Then we define a modified pairing $\lrb{\sigma,\D u^+}_{u_0}^\ast$ by interpreting
$u^+$ as $\max\{u_{\partial\Omega}^{\rm int},(u_0)_{\partial\Omega}^{\rm int}\}$ on
$\partial\Omega$ and extending the $(\Div\,\sigma)$-integral in
\eqref{eq:sigmaDu-Dir} from $\Omega$ to $\overline\Omega$. In other words,
we define the measure $\lrb{\sigma,\D u^+}_{u_0}^\ast$ by setting
\begin{equation}\label{eq:mod-pairing}
  \lrb{\sigma,\D u^+}_{u_0}^\ast
  \coleq\lrb{\sigma,\D u^+}_{u_0}
  +\big[(u{-}u_0)_{\partial\Omega}^{\rm int}\big]_+(-\Div\,\sigma)\ecke\partial\Omega\,.
\end{equation}
With these conventions, we now complement Definition \ref{defi:super-1} as
follows.

\begin{definition}[super-$1$-harmonic function with respect to a
    Dirichlet datum]\label{defi:super-1-Dir}
  For bounded $\Omega$ with \eqref{eq:domain-assum} and
  $u_0\in\W^{1,1}(\Omega)\cap\L^\infty(\Omega)$, we say that
  $u\in\BV(\Omega)\cap\L^\infty(\Omega)$ is weakly super-$1$-harmonic on
  $\overline\Omega$ with respect to $u_0$ if there exists some
  $\sigma\in\Sinfty$ with $\Div\,\sigma\le0$ in ${\mathscr D}'(\Omega)$ such
  that the equality of measures $\lrb{\sigma,\D u^+}_{u_0}^\ast=|\D_{u_0}u|$
  holds on $\overline\Omega$.
\end{definition}

For $\sigma\in\Sinfty$ with $\Div\,\sigma\le0$ in ${\mathscr D}'(\Omega)$ and
$u_0\in\W^{1,1}(\Omega)\cap\L^\infty(\Omega)$,
$u\in\BV(\Omega)\cap\L^\infty(\Omega)$, we get from \eqref{eq:boundary-equality},
\eqref{eq:div-sigma-boundary}, \eqref{eq:mod-pairing}
\[
  \lrb{\sigma,\D u^+}_{u_0}^\ast\ecke\partial\Omega
  =\Big(\big[(u{-}u_0)_{\partial\Omega}^{\rm int}\big]_+
  -\big[(u{-}u_0)_{\partial\Omega}^{\rm int}\big]_-\sigma_{\rm n}^\ast\Big)\H^{n-1}\ecke\partial\Omega\,.
\]
Thus, the boundary condition in Definition \ref{defi:super-1-Dir} is equivalent
to the $\H^{n-1}$-{\ae} equality $\sigma_{\rm n}^\ast\equiv{-}1$ on the boundary
portion
$\{u_{\partial\Omega}^{\rm int}<(u_0)_{\partial\Omega}^{\rm int}\}\cap\partial\Omega$,
while no requirement is made on
$\{u_{\partial\Omega}^{\rm int}\ge(u_0)_{\partial\Omega}^{\rm int}\}\cap\partial\Omega$.
We believe that this is very natural, in particular in the case $n=1$,
where super-$1$-harmonicity of $u$ on an interval ${[a,b]}$ just means that $u$
is increasing up to a certain point and decreasing afterwards, and where
$\sigma_{\rm n}^\ast$ can take the value ${-}1$ at most at \emph{one} endpoint and
only if $u$ is monotone on the open interval ${(a,b)}$.

Another indication that Definition \ref{defi:super-1-Dir} is meaningful is
provided by the next statement, which will also be proved in
\cite{SCHSCH16super1}. We emphasize that the statement does not hold anymore
(not even for $n{=}1$, $u_{0;k}=u_0\equiv0$, and $u_k\in\W_0^{1,1}(\Omega)$) if
one replaces $\lrb{\sigma,\D u^+}_{u_0}^\ast$ with $\lrb{\sigma,\D u^+}_{u_0}$
in the definition.

\begin{theorem}\label{theo:convergence-Dir}
  Suppose that $\Omega$ is bounded with \eqref{eq:domain-assum}, and consider
  weakly super-$1$-harmonic functions $u_k\in\BV(\Omega)\cap\L^\infty(\Omega)$ on
  $\overline\Omega$ with respect to boundary data
  $u_{0;k}\in\W^{1,1}(\Omega)\cap\L^\infty(\Omega)$. If $u_{0;k}$ converges strongly
  in $\W^{1,1}(\Omega)$ and weakly$\ast$ in $\L^\infty(\Omega)$ to some $u_0$, and
  if $u_k$ weak$\ast$ converges to a limit $u$ in $\BV(\Omega)$ and
  $\L^\infty(\Omega)$ such that $u_k\le u$ holds on $\Omega$ for all $k\in\N$,
  then $u$ is weakly super-$1$-harmonic on $\overline\Omega$ with respect to
  $u_0$.
\end{theorem}

\begin{remark}
  In the situation of the theorem, it follows from the previously recorded
  reformulation of the boundary condition that $u$ is also weakly
  super-$1$-harmonic on $\overline\Omega$ with respect to every
  $\widetilde u_0\in\W^{1,1}(\R^n)\cap\L^\infty(\Omega)$ such that
  $\H^{n-1}(\{u_0^\ast\le u_{\partial\Omega}^{\rm int}<\widetilde u_0^\ast\}
  \cap\partial\Omega)=0$. Roughly speaking, this means that the boundary values
  can always be decreased and that they can even be increased as long as the
  trace of $u$ is not traversed. In view of the $1$-dimensional case, we believe
  that this behavior is very reasonable.
\end{remark}

\end{document}